\documentclass[a4paper,draft,reqno,12pt]{amsart}
\usepackage[utf8]{inputenc}
\usepackage{amsmath,amsthm,amssymb}
\usepackage{mathtext}
\usepackage[T1,T2A]{fontenc}

\usepackage{graphicx}
\usepackage{cite}
\usepackage{secdot}

\title{Generalization of the ABC theorem on locally nilpotent derivations}
\author{Veronika Kikteva}
\address{HSE University, Faculty of Computer Science, 11 Pokrovsky Bulvar, Moscow, 109028, Russia}
\email{kiktevanika@gmail.com}
\thanks{The study was implemented in the framework of the Basic Research Program at the National Research University Higher School of Economics in 2023.}
\thanks{This is a preprint of the Work accepted for publication in Siberian Mathematical Journal, 2023, Vol. 64, No 5, pp. 1167-1178, http://pleiades.online/.}
\begin{document}
\maketitle

\theoremstyle{plain}
\newtheorem{thm}{Theorem}[section]
\newtheorem{lem}[thm]{Lemma}
\newtheorem{que}[thm]{Question}
\newtheorem{sled}[thm]{Corollary}
\newtheorem{opr}[thm]{definition}
\theoremstyle{remark}
\newtheorem{re}[thm]{Remark}
\theoremstyle{definition}
\newtheorem{example}[thm]{Example}

\begin{abstract}
We obtain a generalization of the ABC Theorem on locally nilpotent derivations to the case of the polynomials with $m$ monomials such that each variable is included just in one monomial. As applications of this result we provide some construction of rigid and semi-rigid algebras and describe the Makar-Limanov invariant of algebras of a special form.
\end{abstract}

\section{Introduction}

Suppose $\mathbb{K}$ is an algebraically closed field of characteristic zero. Let $B$ be a commutative $\mathbb{K}$-domain. A \textit{derivation} on $B$ is a linear map $D$: $B\to B$ which satisfies the Leibniz rule. A derivation $D$ on $B$ is \textit{locally nilpotent} if for each $b\in B$ there
exists a non-negative integer $n$ such that $D ^n (b)=0$. We denote the set of locally nilpotent derivations on $B$ by $\mathrm{LND}(B)$.

We say that $b_1,b_2\in B$ are \textit{relatively prime} in $B$ if $(b_1) \cap (b_2) = (b_1 b_2)$. This generalizes the definition of relative primeness for nonzero elements in a UFD.

The following theorem is an important result in the theory of locally nilpotent derivations. It states that if some elements satisfy a polynomial equation of a special form, and additional conditions are imposed on them, then these elements are contained in the kernel of every locally nilpotent derivation.

\begin{thm}\normalfont{\cite[ABC Theorem]{fr}}
    \label{ABC}
    Suppose that $x,y,z\in B$ are pairwise relatively prime and satisfy $x^a+y^b+z^c=0$ for some integers $a, b, c \geq 2$. If $a^{-1}+b^{-1}+c^{-1}\leq 1$, then $\mathbb{K}[x,y,z]\subseteq \mathrm{Ker}\ D$ for all $D\in \mathrm{LND}(B)$.
\end{thm}

The ABC Theorem follows from the Mason~--- Stothers Theorem. Let $q$ be a polynomial in one variable over the field $\mathbb{K}$. Denote by $N(q)$ the number of distinct roots of $q$. The Mason~--- Stothers Theorem states the following.

\begin{thm}\normalfont{\cite{Stothers}}
    \label{MSABC}
    Let $a(t),b(t),c(t)$ be relatively prime not all constant polynomials in one variable over $\mathbb{K}$. If $a(t)+b(t)+c(t)=0$ then $$\mathrm{max}\{ \mathrm{deg}\ a(t), \mathrm{deg}\ b(t),\mathrm{deg}\ c(t)\}\leq N (a(t)b(t)c(t))-1.$$
\end{thm}

Theorem~\ref{genMSABC} generalizes the result of Mason and Stothers.

\begin{thm}\normalfont{\cite[Theorem B.1.3]{genABC}}
\label{genMSABC}
Let $n\geq 3$ and $f_1,\dots,f_n$ be not all constant polynomials over $\mathbb{K}$ such that $f_1+\dots+f_n=0$. Assume furthermore that for all $1\leq i_1<\dots<i_s\leq n$ we have $$f_{i_1}+\dots+f_{i_s}=0 \Longrightarrow \mathrm{gcd}(f_{i_1},\dots,f_{i_s})=1.$$
Then $$\max_{1\leq k\leq n} \mathrm{deg}(f_k) \leq (n-2) (N(f_1)+\dots+N(f_n)-1).$$
\end{thm}

The \textit{Makar-Limanov invariant} of a commutative $\mathbb{K}$-domain $B$ is the subalgebra of $B$ defined by $$\mathrm{ML}(B)=\bigcap_{D\in \text{LND}(B)} \text{Ker}\ D.$$ This invariant was introduced by Makar-Limanov in~\cite{ML}, called the ring
of absolute constants or the absolute kernel, and denoted by $\mathrm{AK}(B)$. This invariant has many remarkable properties, see~\cite{fr}. It is easy to see that  $$\mathrm{ML}(\mathbb{K}[X_1,\dots,X_n])=\mathbb{K}.$$ The Makar-Limanov invariant can be used to prove that some varieties are  not isomorphic. In particular, this invariant was used to prove that the Koras-Russell cubic threefold is not isomorphic to the affine space, see~\cite{ML}. It provided a new idea that led to solving the Linearization Problem for $\mathbb{C}^{\times}$-actions on $\mathbb{C}^3$.

A commutative $\mathbb{K}$-domain $B$ is called \textit{rigid} if $\mathrm{ML}(B)=B$. Equivalently, there are no nonzero locally nilpotent derivations on $B$. An affine algebraic variety $X$ is \textit{rigid} if the coordinate ring of $X$ is rigid. A commutative $\mathbb{K}$-domain $B$ is \textit{semi-rigid} if there exists $D\in \text{LND}(B)$ such that $\mathrm{ML}(B)=\text{Ker}\ D$. Recall that two locally nilpotent derivations are equivalent if their kernels coincide. A semi-rigid algebra is either rigid or admits only one nonzero locally nilpotent derivation up to equivalence, see~\cite{fr}.

The Mason~--- Stothers Theorem and the ABC Theorem are helpful in dealing with polynomials over algebraically closed fields. In particular, Fermat's Last Theorem for polynomials follows from Theorem~\ref{MSABC}, see~\cite[\S 7]{lang}.

Using these theorems, the automorphism groups of certain factorial varieties were investigated in~\cite{FinstonMaubach}, and the rigidity of some varieties was proved in~\cite{Chitayat, Crachiola, FinstonMaubachAlmostRigid}. For further applications and generalizations, see~\cite{IV,Gundersen,langDiophantine} and the chapter "Another generalization of Mason’s ABC-theorem" of~\cite{genABC}.

The aim of this paper is to prove a generalization of the ABC Theorem to the case of the polynomials such that each variable is included just in one monomial and to consider some applications of this result. The generalization of the ABC Theorem is stated in Theorem~\ref{mysumeq0}. To prove this statement we use Theorem~\ref{genMSABC}, which contains the generalization of the Mason~--- Stothers Theorem.

Theorem~\ref{mysumeq0} has several applications, some of them are given in the last section of this paper. We construct the new classes of rigid and
semi-rigid algebras, and describe the Makar-Limanov invariant of algebras of a special form.

The second result of this paper is Theorem~\ref{mysumneq0} which generalizes the following result.

\begin{thm}\normalfont{\cite[Theorem 2.50 (a)]{fr}}
Let $D\in \mathrm{LND}(B)$ be nonzero. Suppose $u,v\in \mathrm{Ker}\ D$ and $x,y\in B$ are nonzero, while $a,b\in \mathbb{Z}_{\geq 2}$. Assume
$ux^a+vy^b\neq 0$. Then the condition $D(ux^a+vy^b)=0$ implies $D(x)=D(y)=0$.
\end{thm}
Note that Theorem~\ref{mysumneq0} allows us to answer the following question in a particular case.
\begin{que}\normalfont{\cite[Question 11.9]{fr}}
Let $B$ be an affine $\mathbb{C}$-domain, and let $p\in \mathbb{C}^{[m]}$, with ${m\geq 2}$, be
such that $\delta(p)\neq 0$ for all nonzero $\delta \in \mathrm{LND}(\mathbb{C}^{[m]})$. Suppose that there exist algebraically independent $a_1,\dots,a_m\in B$ and nonzero $D\in \mathrm{LND}(B)$ such that
$D(p(a_1,\dots,a_m))=0$. Does this imply $D(a_i)=0$ for each $i=1,\dots,m$?
\end{que}

The author is grateful to Sergey~Gaifullin and to Ivan~Arzhantsev for constant attention to this work.

\section{Preliminaries}

Let $D$ be a locally nilpotent derivation on a commutative $\mathbb{K}$-domain $B$. Then the derivation $D$ induces the degree function on $B$:
$$\mathrm{deg}_D(b)=\mathrm{min} \{n\in \mathbb{N}\ |\   D^{n+1}(b)=0 \}$$ for nonzero $b\in B$, and $\mathrm{deg}_D(0)= - \infty$.

Recall that a subalgebra $A\subseteq B$ is called \textit{factorially closed} if for all nonzero $x,y\in B$ the condition $xy\in A$ implies $x,y\in A$. The kernel of every locally nilpotent derivation is factorially closed, see~\cite[Principle 1 (a)]{fr}.

An element $t\in B$ is a \textit{local slice} of a derivation $D\in \mathrm{LND}(B)$ if $D^2(t)=0$ and $D(t)\neq 0$. Note that if a derivation $D$ is locally nilpotent and nonzero, then it has a local slice. Indeed, it suffices to choose $b\in B$ such that $D(b)\neq 0$ and put $n:=\mathrm{deg}_D (b)\geq 1$. We have $D^n(b)\neq 0$ and $D^{n+1}(b)=0$. Therefore, $D^{n-1}(b)$ is a local slice for $D$.

Denote by $\mathrm{Quot}(B)$ the quotient field of $B$. Let $S\subseteq B \setminus \{ 0 \}$ be a multiplicatively closed subset. Consider the ring $$S^{-1}B:=\{ ab^{-1}\in \mathrm{Quot}(B) \ | \ a\in B, b\in S \}.$$ The ring $S^{-1}B$ is called the \textit{localization} of $B$ at $S$. If $S=\{ f^i \}_{i\geq 0}$ for some nonzero $f\in B$ then $B_f$ denotes $S^{-1}B$.  

\begin{thm}\normalfont{\cite[Principle 11 (d)]{fr}}
\label{thmslice}
Let $D\in\mathrm{LND}(B)$ be given, and let $t\in B$ be a local slice of $D$. Set $A=\mathrm{Ker}\ D$. Then $B_{D(t)}=A_{D(t)}[t]$.
\end{thm}

Note that if there is a nonzero locally nilpotent derivation $D$ on $B$ then it is possible to embed $B$ in the polynomial ring in a local slice $t$ of $D$ over the quotient field of the kernel of $D$. The following lemma claims that if elements are pairwise relatively prime in $B$ then they are pairwise relatively prime as polynomials in $t$.

\begin{lem}
\label{vzprost}
Let $A$ be a subalgebra of $B$ and let $t\in B$ be a transcendent element over $A$ such that $B\subseteq \mathrm{Quot}(A)[t]$. 
If elements $g_1,\dots,g_m$ are pairwise relatively prime as elements of $B$ then they are pairwise relatively prime as elements of $\mathrm{Quot}(A)[t]$.
\end{lem}
\begin{proof}
Assume that $g_i$ and $g_j$ are relatively prime in $B$ and have a nontrivial common divisor $f$ in $\mathrm{Quot}(A)[t]$. Since $f$ divides $g_i$ in $\mathrm{Quot}(A)[t]$, it follows that $\frac{g_i}{f}$ has the form $$\frac{g_i}{f}=\frac{c_n}{a}t^n+\dots +\frac{c_1}{a}t +\frac{c_0}{a},$$ where $a, c_0,\dots ,c_n\in A$. Similarly, the element $\frac{g_j}{f}$ has the form $$\frac{g_j}{f}=\frac{d_k}{b}t^k+\dots +\frac{d_1}{b}t +\frac{d_0}{b}$$ for some $b,d_0,\dots, d_k\in A$.

Note that $\frac{a g_i}{f},\frac{b g_j}{f}\in B$. Therefore, the element $$q:= \frac{abg_i g_j}{f}=g_i a \frac{b g_j}{f}=g_j b \frac{a g_i}{f}$$  is contained in the intersection of the ideals generated by $g_i$ and $g_j$, i.e., $q\in (g_i)\cap (g_j)$.

If $q$ is contained in the ideal $(g_i g_j)$,
then there exists an element $h\in B$ such that $$q=\frac{abg_i g_j}{f}=g_i g_j h.$$ 
Therefore, we obtain $$\mathrm{deg}_t(a)+\mathrm{deg}_t(b)=0=\mathrm{deg}_t(h)+\mathrm{deg}_t(f).$$ 
Thus we have $\mathrm{deg}_t(f)=0$, because the degree of elements of the algebra $B$ as polynomials in $t$ cannot be negative. Since $f$ is not a constant in $\mathrm{Quot}(A)[t]$, we have that the elements $g_i$ and $g_j$ are not relatively prime. This contradiction proves the lemma.
\end{proof}

\section{The main results}
Let $D$ be a locally nilpotent derivation on the commutative $\mathbb{K}$-domain $B$. Suppose $m,n_1,\dots,n_m$ are positive integers, where $m\geq 3$, while $a_{i}\in \text{Ker}\ D$ are nonzero, $k_{ij}$ are positive integers and elements $b_{ij}\in B$, where $i=1,\dots,m$, and $j=1,\dots,n_i$.

Furthermore, $$f_i:=a_i b_{i1}^{k_{i1}}\dots b_{in_i}^{k_{in_i}} \in B,$$
where $i=1,\dots,m.$

\begin{thm}
\label{mysumeq0}
Suppose that $f_1+\dots+f_m=0$ and elements $f_1,\dots,f_m$ are pairwise relatively prime. Assume that $$\sum_{i=1}^{m} \sum_{j=1}^{n_i} \frac{1}{k_{ij}}\leq \frac{1}{m-2}.$$

Then the element $b_{ij}$ is contained in the kernel of $D$ for all $i=1,\dots,m$ and $j=1,\dots,n_i$.
\end{thm}

\begin{proof}
Denote by $A$ the kernel of $D$ and by $t\in B$ a local slice of $D$. By Theorem~\ref{thmslice} it follows that there exists an embedding $B\subseteq K[t]=K^{[1]}$, where $K=\mathrm{Quot}(A)$, and the degree of elements of $B$ as polynomials in $t$ coincides with the degree function induced by the locally nilpotent derivation $D$: ${\mathrm{deg}_t=\mathrm{deg}_D=:\mathrm{deg}}$. In particular, an element of $B$ is contained in the kernel of $D$ if and only if it is constant as a polynomial in $t$.

Without loss of generality it can be assumed 
that $$\mathrm{deg}(f_1)\geq \dots \geq \mathrm{deg}(f_m).$$ Then for all $i=1,\dots,m$ and $j=1,\dots,n_i$ we have 
$$\mathrm{deg}(b_{ij})=\frac{k_{ij} \mathrm{deg}(b_{ij})}{k_{ij}}\leq \frac{\sum_{\iota=1}^{n_i} k_{i\iota} \mathrm{deg}(b_{i\iota})}{k_{ij}}=\frac{\mathrm{deg}(f_i)}{k_{ij}}\leq \frac{\mathrm{deg}(f_1)}{k_{ij}}.$$ 

The proof is by reductio ad absurdum. Assume that there exists a positive integer $i$, $1\leq i\leq m$, such that $\mathrm{deg}(f_i)>0$. From the conditions of Theorem~\ref{mysumeq0}, elements of any subset $\{f_{i_1},\dots,f_{i_s}\} \subseteq \{f_1,\dots,f_m\}$ are pairwise relatively prime. By Lemma~\ref{vzprost} they are pairwise relatively prime as elements of $K[t]$. Then $\mathrm{gcd}(f_{i_1},\dots,f_{i_s})=1$ for each subset ${\{f_{i_1},\dots,f_{i_s}\} \subseteq \{f_1,\dots,f_m\}}$.

By Theorem~\ref{genMSABC} we obtain: 
$$\mathrm{deg}(f_1)\leq (m-2)(N(f_1)+\dots+N(f_m)-1)\leq $$ $$\leq (m-2)\left(\sum_{i=1}^{m} \sum_{j=1}^{n_i} N(b_{ij})-1\right)\leq (m-2)\left( \sum_{i=1}^{m} \sum_{j=1}^{n_i}\mathrm{deg}(b_{ij})-1\right)\leq $$ $$\leq (m-2)\left( \mathrm{deg}(f_1) \sum_{i=1}^{m} \sum_{j=1}^{n_i} \frac{1}{k_{ij}} -1\right).$$
Therefore, we have $$\mathrm{deg}(f_1)\left( \frac{1}{m-2} - \sum_{i=1}^{m} \sum_{j=1}^{n_i} \frac{1}{k_{ij}} \right)\leq -1<0.$$ This is a contradiction. We have that the element $f_i$ belongs to the kernel of $D$ for all $i=1,\dots,m$. Since the kernel of a locally nilpotent derivation is factorially closed, we obtain that the element $b_{ij}$ is contained in the kernel of~$D$ for any $i=1,\dots,m$ and $j=1,\dots,n_i$.
\end{proof}

\begin{re}
    Theorem~\ref{mysumeq0} generalizes Theorem~\ref{ABC}: put $m=3,\ n_1=n_2=n_3 = 1$ and $a_1=a_2=a_3= 1$. Then the condition $$\sum_{i=1}^{m} \sum_{j=1}^{n_i} \frac{1}{k_{ij}}\leq \frac{1}{m-2}$$ changes to $$a^{-1}+b^{-1}+c^{-1}\leq 1$$ from Theorem~\ref{ABC}.
\end{re}

\begin{thm}
\label{mysumneq0}
Suppose that $f_1+\dots+f_m\in \mathrm{Ker}\ D \backslash \{ 0\}$ and for all $s\leq m$ and ${1\leq i_1<\dots<i_s\leq m}$ we have $$f_{i_1}+\dots+f_{i_s}=0 \Longrightarrow f_{i_1},\dots,f_{i_s}\ \textit{are\ pairwise\ relatively\ prime}.$$ Assume that $$\sum_{i=1}^{m} \sum_{j=1}^{n_i} \frac{1}{k_{ij}} \leq \frac{1}{m-1}.$$

Then the element $b_{ij}$ is in the kernel of $D$ for all $i=1,\dots,m$ and $j=1,\dots,n_i$.
\end{thm}

\begin{proof}
We apply the argument utilized in the proof of Theorem~\ref{mysumeq0}. There exists an embedding $B\subseteq K[t]$, where $K$ is the quotient field of the kernel of the locally nilpotent derivation $D$ and $t$ is a local slice of $D$.

As above, we assume that $$\mathrm{deg}(f_1)\geq \dots \geq \mathrm{deg}(f_m),$$ therefore, 
$$\mathrm{deg}(b_{ij})=\frac{k_{ij} \mathrm{deg}(b_{ij})}{k_{ij}}\leq \frac{\sum_{\iota=1}^{n_i} k_{i\iota} \mathrm{deg}(b_{i\iota})}{k_{ij}}=\frac{\mathrm{deg}(f_i)}{k_{ij}}\leq \frac{\mathrm{deg}(f_1)}{k_{ij}}$$ for all $i=1,\dots,m$ and $j=1,\dots,n_i$. 

We denote $f_{m+1}:=-(f_1+\dots+f_m)$. Note that $f_{m+1}$ is a nonzero element of $K$. Hence, $\mathrm{deg}(f_{m+1})=N(f_{m+1})=0$. We assume again that there exists a positive integer $i$ such that $\mathrm{deg}(f_i)>0$. From the conditions of Theorem~\ref{mysumneq0} if $f_{i_1}+\dots=f_{i_s}=0$ then $f_{i_1},\dots,f_{i_s}$ are pairwise relatively prime in $B$. By Lemma~\ref{vzprost}, for every subset $\{f_{i_1},\dots,f_{i_s}\}\subseteq \{f_1,\dots f_m\}$ such that $f_{i_1}+\dots+f_{i_s}=0$ we have $\mathrm{gcd}(f_{i_1},\dots,f_{i_s})=1$  considered as polynomials in $t$. By Theorem~\ref{genMSABC}, applied to $f_1+\dots+f_m+f_{m+1}=0$, we have
$$\mathrm{deg}(f_1)\leq (m-1)(N(f_1)+\dots+N(f_m)+N(f_{m+1})-1)\leq $$ $$ \leq (m-1)\left(\sum_{i=1}^{m} \sum_{j=1}^{n_i} N(b_{ij})-1\right)\leq  (m-1)\left( \sum_{i=1}^{m} \sum_{j=1}^{n_i}\mathrm{deg}(b_{ij})-1\right)\leq $$ $$ \leq(m-1)\left( \mathrm{deg}(f_1) \sum_{i=1}^{m} \sum_{j=1}^{n_i} \frac{1}{k_{ij}} -1\right).$$
Therefore, we obtain $$\mathrm{deg}(f_1)\left( \frac{1}{m-1} - \sum_{i=1}^{m} \sum_{j=1}^{n_i} \frac{1}{k_{ij}} \right)\leq -1<0.$$ This is a contradiction. We get that $f_1,\dots,f_m$ are in the kernel of $D$. Since the kernel of a locally nilpotent derivation is factorially closed, we get that every $b_{ij}$ is contained in the kernel of $D$.
\end{proof}

\begin{re}
    In case $B$ is a factorial algebra, the condition on $f_1,\dots,f_m$ to be pairwise relatively prime in Theorem~\ref{mysumeq0} can be replaced by the following two conditions: $\mathrm{gcd}(f_1,\dots,f_m)=1$ and for any positive integer $s \leq m$ for all $1\leq i_1<\dots<i_s\leq m$ $$f_{i_1}+\dots+f_{i_s}=0 \Longrightarrow \mathrm{gcd}(f_{i_1},\dots,f_{i_s})=0,$$ because it suffices for applying Theorem~\ref{genMSABC}. This relates Corollary~\ref{RIGIDmysumeq0} if $B/(F)$ is a UFD, and Corollary~\ref{MLmysumeq0}, if $B$ is a UFD.

    In Theorem~\ref{mysumneq0}, if $B$ is a factorial algebra then the condition that $f_{i_1},\dots,f_{i_s}$ are pairwise relatively prime can be replaced by $\mathrm{gcd}(f_{i_1},\dots,f_{i_s})=1$.
\end{re}

\section{Applications}

\subsection{Construction of rigid algebras}

Suppose $m,n_1,\dots,n_m$ are positive integers, where $m\geq 3$, $a_{i}\in \mathbb{K}^{\times}$, $k_{ij}$ are positive integers and $b_{ij}\in B$, where $i=1,\dots,m$ and $j=1,\dots,n_i$. Assume that elements $b_{ij}$ generate $B$ as an algebra.

We denote $$F_i:=a_i b_{i1}^{k_{i1}}\dots b_{in_i}^{k_{in_i}} \in B,$$
where $i=1,\dots,m.$

Theorem~\ref{mysumeq0} implies

\begin{sled}
\label{RIGIDmysumeq0}
Let $F:=F_1+\dots+F_m$ be a prime element of $B$. Denote by~$f_i$ the image of $F_i$ in the factor-algebra $B/(F)$. Assume that $f_1,\dots,f_m$ are pairwise relatively prime and $$\sum_{i=1}^{m} \sum_{j=1}^{n_i} \frac{1}{k_{ij}}\leq \frac{1}{m-2}.$$

Then the algebra $B/(F)$ is rigid.
\end{sled}

To apply Corollary~\ref{RIGIDmysumeq0} it is necessary to check the pairwise relative primeness of those elements whose sum in the factor-algebra equals to zero. The following lemma can be used to prove this.

\begin{lem}
\label{mylemma}
Under the previous notation, if elements $b_{ij}$ are algebraically independent, then $f_1,\dots ,f_m$ are pairwise relatively prime.

\end{lem}
\begin{proof}
Without loss of generality we prove that elements $f_1$ and $f_2$ are relatively prime. It suffices to show that there is an embedding $(f_1) \cap (f_2)\subseteq (f_1 f_2)$. Take an element $g\in (f_1) \cap (f_2)$. Let $G$ be a preimage of $g$ in the polynomial ring $\mathbb{K}[
b_{11},\dots, b_{mn_m}
]$, $$G=H_1 F_1 +P_1 F= H_2 F_2 +P_2 F$$ for some polynomials $H_1,H_2,P_1,P_2\in \mathbb{K}[
b_{11},\dots, b_{mn_m}
].$ 
Hence,
    $$H_1 F_1 + (P_1 - P_2) F = H_2 F_2 = G - P_2 F.\eqno{(1)}$$

Monomials in the polynomial $P_1-P_2$ can be divided into two parts. The first part contains the terms divisible by $F_2$ and the second part consists of the rest monomials. The same is true for the polynomial $H_1$.

Therefore, these polynomials can be represented in the following form: 
    $$P_1 - P_2 = P_0 F_2 + \widetilde{P},\ H_1 = H_0 F_2 + \widetilde{H},  \eqno{(2)}$$
where $\widetilde{P}$ and $\widetilde{H}$ have no monomials divisible by $F_2$.

Combining (2) and (1), we obtain
$$F_2\ |\ (H_0 F_2 + \widetilde{H}) F_1 + (P_0 F_2 + \widetilde{P})F.$$ Subtracting elements divisible by $F_2$, we have $$F_2\ |\ \widetilde{H} F_1 + \widetilde{P} (F_1+F_3+\dots+F_m),$$ while $F_1$ and $F_1+F_3+\dots+F_m$ do not contain variables from $F_2$. Thus, we get $$\widetilde{H} F_1 + \widetilde{P} (F_1+F_3+\dots+F_m)=0,$$ and $F_1\ |\ \widetilde{P}$. Let $\widetilde{P}=\widehat{P}_0 F_1$. Then from decomposition (2) we obtain $P_1-P_2 = \widehat{P}_0 F_1 + P_0 F_2$. 

Using (1) we get that the polynomial $F_2$ divides the polynomial
$$H_1 F_1 + (\widehat{P}_0 F_1 + P_0 F_2) F = H_2 F_2 = G - P_2 F=$$
$$= (H_1 + \widehat{P}_0 F)F_1 + P_0 F_2 F,$$ hence, $F_2\ |\ H_1 + \widehat{P}_0 F$. So, $G - P_2 F - P_0 F_2 F \in (F_1 F_2)$ and $g\in (f_1 f_2)$. This completes the proof of the relatively primeness of $f_1$ and $f_2$.
\end{proof}

\begin{re}
In case $m=3$ the statement of Lemma~\ref{mylemma} follows from~\cite{HH}. Recall the necessary definitions. Let $R$ be an algebra graded by a finitely generated abelian group $K$. An element $r\in R$ is called a $K$\textit{-prime} if $r$ is a homogeneous nonzero nonunit such that if $r$ divides a product of some homogeneous elements, then it divides one of the factors. An algebra $R$ is called \textit{factorially graded} if every homogeneous nonzero nonunit is a product of $K$-primes. This decomposition is unique up to multiplication by a unit and permutation of factors.

Let us prove Lemma~\ref{mylemma} in case $m=3$ using results of~\cite{HH}. We denote by $\widetilde{b}_{ij}$ the image of $b_{ij}$ in the factor-algebra $B/(F)$. Note that $B/(F)$ is factorially graded with respect to the abelian finitely generated group $K$ (see~\cite[Theorem~10.1(i)]{HH}), and $\widetilde{b}_{ij}$ are pairwise nonassociated $K$-primes (see~\cite[Theorem~10.4(i)]{HH}). If a homogeneous element $g\in B/(F)$ is contained in the intersection of the ideals $(f_i)$ and $(f_j)$, then it has the form $$g=a_i \widetilde{b}_{i1}^{k_{i1}}\dots \widetilde{b}_{in_i}^{k_{in_i}} h_i = a_j \widetilde{b}_{j1}^{k_{j1}}\dots \widetilde{b}_{jn_j}^{k_{jn_j}} h_j$$ for some elements $h_i, h_j\in B/(F)$. Since a decomposition into $K$-primes is unique, we obtain $f_i|h_j$ and $f_j|h_i$. Therefore, the element $g$ belongs to the ideal~$(f_i f_j)$. Every element of the factor-algebra $B/(F)$ can be represented as the sum of homogeneous elements, hence, $(f_i)\cap (f_j)=(f_i f_j)$. The relative primeness of $f_i$ and $f_j$ is proved.
\end{re}

\begin{example}
\label{trinom}
(Trinomial hypersurfaces) 
Suppose $n_0,n_1$ and $n_2$ are positive integers, $l_{ij}$ are non-negative integers, where $0\leq i \leq 2$ and $1\leq j \leq n_i$.
We denote $$T_i^{l_i}:=T_{i1}^{l_{i1}}\dots T_{in_i}^{l_{in_i}}\in \mathbb{K}[
T_{ij};\ 0\leq i\leq 2,\ 1\leq j\leq n_i
],$$ where $i=0,1,2$. Let $X$ be a trinomial hypersurface with coordinate ring 
$$\mathbb{K}[X] \simeq \mathbb{K}[
T_{01},\dots, T_{2n_2}
]/(
T_0^{l_0}+T_1^{l_1}+T_2^{l_2}).$$
We assume that $\sum l_{ij}^{-1}\leq 1$. By Lemma~\ref{mylemma} we have that the images of $T_0^{l_0}$, $T_1^{l_1}$ and $T_2^{l_2}$ in the factor-algebra are relatively prime. Hence, the algebra $\mathbb{K}[X]$ satisfies the conditions of Corollary~\ref{RIGIDmysumeq0}. By Corollary~\ref{RIGIDmysumeq0} we obtain that the hypersurface $X$ is rigid. 

In this example our result confirms an already known fact. We denote ${d_i:=\text{gcd}(l_{i1},\dots,l_{in_i})}$. If $d_i$ are relatively prime, then we have a factorial trinomial hypersurface, its rigidity follows from ~\cite{Ar16}. Otherwise rigidity of $X$ follows from~\cite{SA}.

In particular, a hypersurface given by the equation $$\{X_1^6X_2^7+Y_1^8Y_2^9+Z_1^{10}Z_2^{11}=0\}$$ in the 6-dimensional affine space with coordinates $X_1$, $X_2$, $Y_1$, $Y_2$, $Z_1$ and $Z_2$ is rigid by Corollary~\ref{RIGIDmysumeq0}. The relative primeness of the images of $X_1^6X_2^7,\ Y_1^8Y_2^9$ and $Z_1^{10}Z_2^{11}$ in the factor-algebra $$\mathbb{K}[X_1,X_2,Y_1,Y_2,Z_1,Z_2]/(X_1^6X_2^7+Y_1^8Y_2^9+Z_1^{10}Z_2^{11})$$ follows from Lemma~\ref{mylemma}.  It is easy to see that the condition $$\frac{1}{6}+\frac{1}{7}+\frac{1}{8}+\frac{1}{9}+\frac{1}{10}+\frac{1}{11}\leq 1$$ holds.
\end{example}

\begin{example} (Trinomial varieties) The class of trinomial varieties is more general than the class of trinomial hypersurfaces. Trinomial varieties are algebraic varieties given by аn appropriate system of trinomials. We use the notation of~\cite{HH}.

For a positive integer $r\geq 1$ let $A=(a_0,\dots,a_r)$ be a sequence of vectors $a_i=(b_i,c_i)\in\mathbb{K}^2$ such that the pair of vectors $(a_i,a_k)$ is linearly independent for any $i\neq k$. Suppose $\mathfrak{n}=(n_0,\dots, n_r)$ is a sequence of positive integers, $L=(l_{ij})$ is a set of positive integers, where $0\leq i\leq r$ and $1\leq j\leq n_i$. For any $0\leq i \leq r$ we define a monomial $$T_i^{l_i}:= T_{i1}^{l_{i1}}\dots T_{in_i}^{l_{in_i}}\in S:= \mathbb{K}[T_{ij};\ 0\leq i \leq r, 1\leq j\leq n_i].$$ Given any pair of indices $0\leq i,j \leq r$ put $\alpha_{ij}:=\mathrm{det}(a_i,a_j)=b_i c_j - b_j c_i$, and for any triple of indices $0\leq i< j< k\leq r$ let $$g_{i,j,k}:=\alpha_{jk}T_i^{l_i} + \alpha_{ki}T_j^{l_j} +\alpha_{ij}T_k^{l_k}\in S.$$
A \textit{trinomial variety} is an affine algebraic variety $X$ with coordinate ring $$\mathbb{K}[X] \simeq R(A,\mathfrak{n}, L):= \mathbb{K}[T_{ij},\ 0\leq i\leq r,\ 1
\leq j \leq n_i]/(g_{i,i+1,i+2};\ 0\leq i\leq r-2).$$
We denote by $t_{ij}$ the image of $T_{ij}$ in the factor-algebra $R(A, \mathfrak{n}, L)$. 
In addition we assume that for every trinomial the sum of the inverse values to all powers of the variables included in this trinomial is less than or equal to 1. To apply Corollary~\ref{RIGIDmysumeq0} it remains to show the relative primeness of $t_i^{l_i}=t_{i1}^{l_{i1}}\dots t_{in_i}^{l_{in_i}}$. This follows from the fact that $\mathbb{K}[X]$ is factorially graded, see~\cite[Theorem~1.1(i)]{HH}.

A criterion for a trinomial variety to be factorial is that numbers $d_i=\text{gcd}(l_{i1},\dots,l_{in_i})$ are pairwise relatively prime, see~\cite[Theorem~1.1(ii)]{HH}. A criterion for a factorial trinomial variety to be rigid is obtained in~\cite{SA}. Any criterion for a non-factorial trinomial variety to be rigid is unknown yet. 

It is easy to construct a non-factorial trinomial variety that is rigid by Theorem~\ref{mysumeq0}. For example, a variety $X$ with coordinate ring $$\mathbb{K}[X]\simeq \mathbb{K}[X_1,X_2,Y,Z_1,Z_2,W_1,W_2]/(X_1^6 X_2^{12}+Y^7 +Z_1^6 Z_2^9,\ X_1^6 X_2^{12}+2Y^7 +W_1^8 W_2^9)$$ satisfies these conditions.
Denote by $x_1,x_2,y,z_1,z_2,w_1$ and $w_2$ the images of the elements $X_1,X_2,Y,Z_1,Z_2,W_1$ and $W_2$ in the factor-algebra.
Since $\mathbb{K}[X]$ is factorially graded, the elements $$x_1^6 x_2^{12}, y^7, z_1^6 z_2^9\ \mathrm{and}\ x_1^6 x_2^{12}, 2y^7, w_1^8 w_2^9$$ are pairwise relatively prime. It is clear that $$\frac{1}{6}+\frac{1}{12}+\frac{1}{7}+\frac{1}{6}+\frac{1}{9}\leq 1 \ \mathrm{and}\  \frac{1}{6}+\frac{1}{12}+\frac{1}{7}+\frac{1}{8}+\frac{1}{9}\leq 1.$$

\end{example}

\begin{example}($m$-term hypersurfaces) 
An affine algebraic variety is called an
\textit{$m$-term hypersurface} if its coordinate ring is isomorphic to a polynomial ring factorized by an ideal generated by a polynomial with $m$ monomials.

In case $b_{ij}$ are algebraically independent and $B$ is a polynomial ring, elements $f_1,\dots,f_m$ are pairwise relatively prime by Lemma~\ref{mylemma}. Assume that $$\sum k_{ij}^{-1}\leq \frac{1}{m-2}.$$ Then the $m$-term hypersurface with coordinate ring isomorphic to $B/(F)$ is rigid by Corollary~\ref{RIGIDmysumeq0}.

For example, the factor-algebra 
$$B=\mathbb{K}[X,Y,Z,V,W]/(X^{10} + Y^{10} Z^{11} + V^{10}+W^{10})$$ is rigid. The pairwise relative primeness of the images of the monomials $X^{10},Y^{10}Z^{11},V^{10}$ and $W^{10}$ in the factor-algebra follows from Lemma~\ref{mylemma}. The condition $$\frac{1}{10}+\frac{1}{10}+\frac{1}{11}+\frac{1}{10}+\frac{1}{10}\leq \frac{1}{2}$$ holds.
\end{example}

\begin{example}
Suppose $F_1,F_2,F_3$ and $F$ are polynomials in variables $X,Y,Z,V$ and $W$ over the field $\mathbb{K}$: $$F_1:=XY+Z,\ F_2:=Y^6 +Z^6 +W^6+ V,\ F_3:=Y^6 +Z^6 +W^6- V,$$ we denote $F:=F_1^3+F_2^3+F_3^3$. An algebra $B$ is given by
$$B=\mathbb{K}[X,Y,Z,V,W]/(F).$$ Let $f_1, f_2$ and $f_3$ be the images of $F_1, F_2$ and $F_3$ in the factor-algebra.

Let us prove that the algebra $B$ is rigid. It is necessary to check the pairwise relative primeness of the elements $f_1^3, f_2^3$ and $f_3^3$. It follows from the pairwise relative primeness of the elements $f_1, f_2$ and $f_3$, because if some elements $a,b\in B$ are relatively prime and $m,n \in \mathbb{Z}_{\geq 1}$, then elements $a^m$ and $b^n$ are relatively prime too, see~\cite{fr}.

Let $g\in (f_1) \cap (f_2)$. Consider a preimage of $g$ in the polynomial ring. It has the form $$G=H_1 F_1 + P_1 F=H_2 F_2 + P_2 F$$ for some polynomials $H_1, H_2, P_1$ and $P_2$. 

Therefore, we have
    $$G-P_2 F=H_1 F_1 + (P_1-P_2)F = H_2 F_2.\eqno{(3)}$$

We reduce all parts of equation~(3) by the polynomial $F_2$ according to the lexicographic order $V>W>X>Y>Z$. This means that we replace $V$ with $-Y^6-Z^6-W^6$. We get that $F_1$ does not change, $F_2$ reduces to zero, and $F$ changes to $F_1^3 + 8(Y^6+Z^6+W^6)^3$. Denote by $\widetilde{H}_1,\ \widetilde{H}_2,\ \widetilde{P}$ and $\widetilde{G}$ the images of the polynomials $H_1,\ H_2,\ P_1-P_2$ and $G$ under the reduction. From~(3) it follows that 
$$\widetilde{H}_1 F_1 + \widetilde{P}(F_1^3 + 8(Y^6+Z^6+W^6)^3) = 0,$$ hence $F_1\ |\ \widetilde{P}$. The polynomial $\widetilde{P}$ is obtained from $P_1-P_2$ by the reduction using $F_2$. Therefore, the polynomial $P_1-P_2$ has the form $\widehat{P}_1 F_1+\widehat{P}_2 F_2$ for some polynomials $\widehat{P}_1$ and $\widehat{P}_2$.

Combining with~(3), we get
$$ G-P_2 F=H_1 F_1 + (\widehat{P}_1 F_1+\widehat{P}_2 F_2)F = H_2 F_2,$$ whence 
$F_2\ |\ H_1 F_1+ \widehat{P}_1 F_1 F$. Therefore, $F_2\ |\ H_1+ \widehat{P}_1 F$ and $G-P_2 F - \widehat{P}_2 F_2 F \in (F_1 F_2) $. We obtain $g\in (f_1 f_2)$. The relative primeness of the elements $f_1$ and $f_2$ is proved. The relative primeness of $f_1$ and $f_3$ can be proved in the same way. 

It remains to show that the elements $f_2$ and $f_3$ are relatively prime. We consider an element $g$ belonging to the intersection of the ideals generated by $f_2$ and $f_3$, i.e., ${g\in (f_2)\cap (f_3)}$. Any preimage of $g$ in the polynomial ring has the form 
$$G=H_2 F_2 + P_2 F = H_3 F_3 + P_3 F \eqno{(4)}$$for some polynomials $H_2,H_3,P_2$ and $P_3$.
Again we reduce all parts of equation~(4) by the polynomial $F$ according to the lexicographic order ${X>Y>Z>V>W}$. We obtain that the polynomials $F_2$ and $F_3$ do not change, because these polynomials are independent of the variable $X$, and $F$ reduces to zero. Denote by $\widetilde{H}_2,\widetilde{H}_3$ and $\widetilde{G}$ the images of the polynomials $H_2,H_3$ and $G$ under the reduction. Therefore, from~(4) we have
$$\widetilde{G}=\widetilde{H}_2 F_2= \widetilde{H}_3 F_3.$$

The polynomials $F_2$ and $F_3$ are relatively prime. Indeed, if these polynomials have a nontrivial common divisor, then the elements $Y^6+Z^6+W^6$ and $V$ have a nontrivial common divisor too, this is a contradiction.

Hence, from $\widetilde{H}_2 F_2= \widetilde{H}_3 F_3$ it follows that $\widetilde{H}_2 \in (F_3)$, whence $\widetilde{G}\in (F_2 F_3)$. Therefore, the image of $\widetilde{G}$ in the factor-algebra is $\widetilde{g}\in (f_2 f_3)$. Note also that $\widetilde{g}=g$, because $\widetilde{G}$ is obtained from $G$ by the reduction using $F$. This concludes the proof of the relative primeness of the elements $f_2$ and $f_3$.

It follows from Theorem~\ref{ABC} that for every derivation $D\in \text{LND}(B)$ we have
$$xy+z,\ y^6 +z^6 +w^6+ v,\ y^6 +z^6 +w^6-v\in \text{Ker}\ D.$$ Therefore, we obtain $v,\ y^6 +z^6 +w^6\in \text{Ker}\ D$. By Theorem~\ref{mysumneq0} we have $y,z,w\in \text{Ker}\ D$. Since the kernel of a locally nilpotent derivation is factorially closed, we get $x\in \text{Ker}\ D$. All generators of the algebra $B$ are in the kernel of each locally nilpotent derivation on $B$. It follows that the algebra $B$ is rigid.
\end{example}

\subsection{Construction of semi-rigid algebras}
Some examples of semi-rigid algebras can be obtained by the following theorem, see~\cite[Theorem 2.24]{fr} and~\cite{MLLND}.
\begin{thm}[Semi-Rigidity Theorem, L. Makar-Limanov] 
\label{semirid}
If $A$ is a rigid commutative $\mathbb{K}$-domain
of finite transcendence degree over $\mathbb{K}$, then $A^{[1]}$ is semi-rigid, where $A^{[1]}$ denotes a polynomial ring in one variable over $A$.
\end{thm}

\begin{example}
Consider the algebra $$B=\mathbb{K}[X,Y,Z,V,W]/((X-Y)^4 + V^4 W^5 +Z^4).$$ Let us change variables by $\widetilde{X}=X-Y,\ \widetilde{Y}=X+Y$. Then
$$B=\mathbb{K}[\widetilde{X},\widetilde{Y},Z,V,W]/(\widetilde{X}^4+V^4 W^5 + Z^4)=\left( \mathbb{K}[\widetilde{X},Z,V,W]/(\widetilde{X}^4+V^4 W^5 + Z^4) \right)[\widetilde{Y}].$$
By Corollary~\ref{RIGIDmysumeq0}, the algebra $$\mathbb{K}[\widetilde{X},Z,V,W]/(\widetilde{X}^4+V^4 W^5 + Z^4)$$ is rigid. Hence, by Theorem~\ref{semirid}, the algebra $B$ is semi-rigid. 
\end{example}
Also, some semi-rigid algebras can be obtained by Theorem~\ref{mysumeq0}. 

\begin{example}
Consider the polynomials $$F_1:=XW-YV,\ F_2:=V^4 W^5,\ F_3:= Z,\ F:=F_1^4 + F_2 +F_3^4 \in \mathbb{K}[X,Y,Z,V,W]$$ and the factor-algebra $$B=\mathbb{K}[X,Y,Z,V,W]/(F).$$ Let $f_1,f_2$ and $f_3$ be the images of the polynomials $F_1,F_2$ and $F_3$ in the factor-algebra $B$. Let us show that the algebra $B$ is semi-rigid. To prove the pairwise relative primeness of $f_1^4,f_2$ and $f_3^4$ it is sufficient to prove that $f_1,f_2$ and $f_3$ are pairwise relatively prime.

Let us prove that $f_1$ and $f_2$ are relatively prime. Let $g\in (f_1)\cap (f_2)$, then every preimage of $g$ in the polynomial ring $G$ has the form $$G=H_1 F_1 + P_1 F = H_2 F_2 + P_2 F$$ for some polynomials $H_1,H_2,P_1$ and $P_2$. We reduce this equation by $F$, assuming that $Z$ is a maximal variable in a lexicographic order. By tilde we denote images of polynomials under the reduction. We obtain $$\widetilde{G}=\widetilde{H}_1 F_1=\widetilde{H}_2 F_2.$$
The polynomials $F_1$ and $F_2$ are relatively prime, therefore, $\widetilde{G} \in (F_1 F_2)$. The image of $\widetilde{G}$ in the factor-algebra $B$ belongs to the ideal $(f_1 f_2)$ and coincides with $g$. Hence $f_1$ and $f_2$ are relatively prime. Similarly, the relative primeness of $f_2$ and $f_3$ can be proved by a reduction according to a lexicographic order with $X$ as the maximal variable.

To prove that $f_1$ and $f_3$ are relatively prime take an element $g\in (f_1)\cap (f_3)$ and its preimage $G$ in the polynomial ring: $$G=H_1 F_1  + P_1 F = H_3 F_3  + P_3 F$$ for some polynomials $H_1,H_3,P_1,P_3$. We have 
$$ H_1 F_1 +(P_1 - P_3) F = H_3 F_3.$$  We consider the residues modulo $F_3 = Z$ of all parts of this equation.
The residues of $H_1$ and $P_1-P_3$ we denote by $\widetilde{H}_1$ and $\widetilde{P}$. We obtain
$$\widetilde{H}_1 F_1  + \widetilde{P} (F_1^4 + F_2) =0,$$
whence $F_1\ |\ \widetilde{P}$. Therefore, for some polynomials $\widehat{P}_1,\widehat{P}_3$ we get $$P_1-P_3 = \widehat{P}_1 F_1 + \widehat{P}_3 F_3.$$
Hence, $$G-P_3 F= H_1 F_1 + (\widehat{P}_1 F_1 + \widehat{P}_3 F_3)F = H_3 F_3$$ and $H_1+\widehat{P}_1 F \in (F_3)$. We get $G-P_3 F - \widehat{P}_3 F_3 F\in (F_1 F_3)$ and $g\in(f_1 f_3)$, which concludes the proof of the relative primeness of $f_1$ and $f_3$.

We denote by $x,y,z,v$ and $w$ the images of $X,Y,Z,V$ and $W$ in the factor-algebra $B$. By Theorem~\ref{mysumeq0} for every $D\in \text{LND}(B)$ we have $xw-yv,v,w,z\in \text{Ker}\ D$. Therefore, we obtain $$\mathbb{K}[xw-yv,v,w,z]\subseteq \mathrm{ML}(B).$$ There exists a derivation $\widetilde{D}$ with kernel $\mathbb{K}[xw-yv,v,w,z]$: $$\widetilde{D}(x)=v,\ \widetilde{D}(y)=w,\ \widetilde{D}(v)=\widetilde{D}(w)=\widetilde{D}(z)=0.$$ Hence, $\text{Ker}\ \widetilde{D}=\mathrm{ML}(B)$, and the algebra $B$ is semi-rigid by definition.
\end{example}

\subsection{Stable Makar-Limanov invariant}
The above results can be reformulated in terms of the Makar-Limanov invariant. Suppose $m,n_1,\dots,n_m$, are positive integers, where $m\geq 3$, $a_{i}\in \mathbb{K}$ are nonzero elements, $k_{ij}$ are positive integers and $b_{ij}\in B$, where $i=1,\dots,m$ and $j=1,\dots,n_i$.

We denote $$f_i:=a_i b_{i1}^{k_{i1}}\dots b_{in_i}^{k_{in_i}} \in B,$$
where $i=1,\dots,m.$

By Theorem~\ref{mysumeq0} we have

\begin{sled}
\label{MLmysumeq0}
Assume that $f_1+\dots+f_m=0$, and elements $f_1,\dots,f_m$ are pairwise relatively prime. Suppose that $$\sum_{i=1}^{m} \sum_{j=1}^{n_i} \frac{1}{k_{ij}}\leq \frac{1}{m-2}.$$

Then $\mathbb{K}[b_{11},\dots,b_{mn_m}] \subseteq \mathrm{ML}(B)$.
\end{sled}

Let $X$ be a rigid affine variety. Theorem~\ref{semirid} states that $\mathrm{ML}(X\times \mathbb{K})=\mathbb{K}[X]$.  It follows 
easily that $\mathrm{ML}(X\times \mathbb{K}^n)\subseteq \mathbb{K}[X]$. 
An example of a rigid surface $X$ such that ${\mathrm{ML}(X\times \mathbb{K}^2)\neq \mathbb{K}[X]}$ is given in~\cite{D}. 

Consider the chain of subalgebras
$$\mathrm{ML}(X)\supseteq \mathrm{ML}(X\times \mathbb{K})\supseteq \mathrm{ML}(X\times \mathbb{K}^2)\supseteq \dots .\eqno{(5)}$$
Define the {\it stable Makar-Limanov invariant} as the subalgebra  equal to the intersection of the elements of (5). We denote it by $\mathrm{SML}(X)$. This term was suggested by Sergey Gaifullin.

\begin{re}
Since $\mathrm{ML}(X\times\mathbb{K}^n)$ is an algebraically closed subalgebra in the algebra ${\mathbb{K}[X\times\mathbb{K}^n]}$, there exist finitely many steps in which chain (5) is stabilized, i.e., there is a positive integer $r$ such that $\mathrm{ML}(X\times\mathbb{K}^r)=\mathrm{SML}(X)$.
\end{re}

Let us remark that elements, belonging to $\mathrm{ML}(X)$ by Corollary~\ref{MLmysumeq0}, are contained in $\mathrm{SML}(X)$. 

\begin{example}
 Let $X=\{X_1^6X_2^7+Y_1^8Y_2^9+Z_1^{10}Z_2^{11}=0\}$, see Example~\ref{trinom}. By Corollary~\ref{MLmysumeq0} it follows that 
 $$\mathrm{SML}(X)=\mathbb{K}[X].$$
\end{example}

\end{document}